\documentclass[a4paper,11pt]{amsart} 
\usepackage{latexsym,amscd,amssymb,url}
\usepackage{enumitem}
\usepackage{mathtools}
\usepackage{comment}

\usepackage[letterpaper]{geometry}

\usepackage[all,cmtip]{xy}  
\usepackage{graphicx}
\usepackage{xcolor}
\definecolor{dblue}{rgb}{0,0,.6}
\usepackage[colorlinks=true, linkcolor=dblue, citecolor=dblue, filecolor = dblue, menucolor = dblue, urlcolor = dblue]{hyperref}
\usepackage{tikz}
\usetikzlibrary{cd}

\numberwithin{equation}{section}

\newtheorem{theorem}{Theorem}[section]

\theoremstyle{plain} 

\newtheorem{corollary}[theorem]{Corollary}
\newtheorem{definition}[theorem]{Definition}
\newtheorem{lemma}[theorem]{Lemma}
\newtheorem{proposition}[theorem]{Proposition}

\theoremstyle{definition} 

\newtheorem{example}[theorem]{Example}

\newtheorem{remark}[theorem]{Remark}

\newcommand{\del}{\partial}
\newcommand{\Z}{\mathbb Z}

\newcommand{\N}{\mathbb N}

\newcommand{\rk}{\operatorname{rk}}

\newcommand{\Alb}{\operatorname{Alb}}

\newcommand{\CH}{\operatorname{CH}}

\newcommand{\F}{\mathbb F}

\newcommand{\mkg}{m_{k}(g)}
\newcommand{\mgg}{m_{g-1}(g)}

\newcommand{\dashedlongrightarrow}{\xymatrix@1@=15pt{\ar@{-->}[r]&}}
\renewcommand{\longrightarrow}{\xymatrix@1@=15pt{\ar[r]&}}
\renewcommand{\mapsto}{\xymatrix@1@=15pt{\ar@{|->}[r]&}}
\renewcommand{\twoheadrightarrow}{\xymatrix@1@=15pt{\ar@{->>}[r]&}}
\newcommand{\hooklongrightarrow}{\xymatrix@1@=15pt{\ar@{^(->}[r]&}}
\newcommand{\congpf}{\xymatrix@1@=15pt{\ar[r]^-\sim&}}

\begin{document}   
\title[On the degree of subvarieties on abelian varieties]{On the degree of subvarieties on abelian varieties}


\author[Engel]{Philip Engel}
\address{Department of Mathematics, Statistics, 
and Computer Science, University of Illinois 
Chicago (UIC),
851 S Morgan St, Chicago, IL 60607, USA}
\email{pengel@uic.edu} 

\author[Schreieder]{Stefan Schreieder}
\address{Leibniz University Hannover, 
Institute of Algebraic Geometry, 
Welfengarten 1, 30167 Hannover, Germany.}
\email{schreieder@math.uni-hannover.de}

\date{\today}



\maketitle
 
\begin{abstract}  
Let $(X,\Theta)$ be a very general principally polarized abelian variety of dimension $g$, and consider the minimal cohomology class $\theta_k=[\Theta]^k/k!$ for $k<g$.
We show that the minimal positive multiple of $\theta_k$ which is algebraic is divisible by all primes $p\leq (k+1)/2$. 
In particular, these minimal multiples grow exponentially with $k$.
Our main result follows from \cite{EGFS} together with a new combinatorial result about $\F_p$-solutions of certain graphic matroids in their own Albanese graphs. 
\end{abstract}

\section{Introduction} 

We work over the field of complex numbers.
Let $(X,\Theta)$ be a principally polarized abelian variety and consider the minimal class
$$
\theta_k=[\Theta]^k/k!\in H^{2k}(X,\Z)
$$
which is an indivisible Hodge class on $X$.
It was recently shown in \cite{EGFS} that the above class is in general not algebraic if $g\geq 4$ and $2\leq k\leq g-1$.

The purpose of this note is to prove the following quantitative result.
 
\begin{theorem} \label{thm:deg(C)}
Let $(X,\Theta)$ be a very general principally polarized abelian variety of dimension $g$.
Let $Z\subset X$ be a closed subvariety of codimension $1\leq k< g$ and write $[Z]=m\cdot \theta_{k}$ with $m\in \N$.
Then 
$$
p\mid m\qquad \text{for all primes $p\leq (k+1)/2$.}
$$
\end{theorem}

Since $\theta_k\theta_{g-k}=\binom{g}{k}$, the above theorem is equivalent to saying that the intersection number $Z\cdot \theta_{g-k}$ is divisible by the product 
of $\binom{g}{k}$ with all primes less than or equal to $\frac{k+1}{2}$.

For positive integers $k,g$ with $k<g$, we define the positive integer
$$
 \mkg 
\coloneq \min\left\{ m\in \Z_{\geq 1}\mid m\cdot \theta_k \textrm{ is algebraic for very general
$(X,\Theta)\in \mathcal A_g$}  \right\}.
$$
In particular, $N_g\coloneq \mgg   $ is the smallest positive multiple of the minimal curve class that is algebraic on a very general 
$(X,\Theta)\in \mathcal A_g$, cf.~\cite{EGFS-A6}.
Since $k!\cdot \theta_k=[\Theta^k]$ is algebraic, an upper bound on $\mkg $ is given by the condition
$$
\mkg  \mid k! .
$$ 
In \cite{EGFS,EGFS-A6} it has been shown that $2\mid \mkg $ for all $g\geq 4$ and $2\leq k\leq g-1$; moreover, $m_3(4)=m_4(5)=2$, $m_5(6)=6$, and $ 6\mid \mkg $ for $5\leq k<g$.
Unboundedness of $\mkg $ remained open.
By linearity, Theorem \ref{thm:deg(C)} also applies to algebraic cycles.
We thus obtain the following corollary, where we recall Chebyshev's theta function:
$$
\vartheta(x)\coloneq \sum_{p\leq x} \log(p), \qquad \text{where the sum runs through all primes $p\leq x$.}
$$

\begin{corollary} \label{cor:thm:deg(C)}
Let $k<g$ be positive integers.
Then  
$$ 
e^{\vartheta\left(\frac{k+1}{2}\right)} \mid \mkg  . 
$$ 
\end{corollary}


By (a version of) the prime number theorem, we have
$ 
\vartheta(x)=x+o(x),
$ 
see e.g.~\cite[Theorem 4.4]{Apostol}.
Using more recent results of Dusart \cite{Dusart}, one can also give effective lower bounds on $\vartheta(x)$ and hence on $\mkg $ as follows (see Corollary \ref{cor:Dusart}):
\begin{align} \label{eq:exponential-1}
\mkg >\exp\!\left(
\frac {k+1}{2} \left(1-\frac{151.3}{\left(\log\left(\frac{k+1}{2}\right)\right)^4}\right) \right) \qquad \text{for all $3\leq k<g$}.
\end{align}
Using a similar estimate of Dusart \cite{Dusart}, we obtain the following simpler (but weaker) bound:
\begin{align} \label{eq:exponential-2}
\mkg >
e^{\frac{k+1}{3}}
\qquad \text{for all $33 \leq k<g$.}
\end{align}

In \cite{debarre-1,debarre-2}, Debarre studied the minimal multiple $e_{g-1}(g)$ of the minimal curve class $\theta_{g-1}$ which is represented by an effective curve for all $(X,\Theta)\in \mathcal A_g$.
In \cite[Section 6]{debarre-2}, Debarre shows that his previous work \cite{debarre-1} together with work of Pirola \cite{pirola} and Pareschi--Popa \cite{pareschi-popa} implies  
$$
e_{g-1}(g)\geq \sqrt{g-1}-\frac{1}{2}.
$$
Clearly, $\mgg    \mid e_{g-1}(g)$.
Therefore, Theorem \ref{thm:deg(C)} and Corollary \ref{cor:thm:deg(C)} 
imply that $e_{g-1}(g)$ is divisible by all primes $p\leq g/2$, 
and so $e_{g-1}(g)$ grows exponentially in $g$:

\begin{corollary} \label{cor:Debarre}
We have 
$$
e_{g-1}(g)\geq \mgg   \geq e^{\vartheta(\frac{g}{2})}=e^{\frac{g}{2}+o(\frac{g}{2})} .
$$
\end{corollary}

Our approach relies on \cite{EGFS,EGFS-A6}.
To prove Theorem \ref{thm:deg(C)}, 
we first deal with the case where $Z\subset X$ is a curve; results for higher codimension $k$ cycles follow via suitable degeneration arguments. 
In \cite{EGFS}, divisibility constraints 
(and hence lower bounds)
for algebraic multiples 
of the minimal curve class 
were reduced to a purely combinatorial problem for regular matroids of rank $g$, which we explain next.

Let $\underline R$ be a regular matroid on the ground set $S$ and with integral realization $S\to U^\ast$, where $U$ denotes some free $\Z$-module of rank $g = \rk(\underline R)$.
Dualizing the integral realization, we obtain an embedding $U\hookrightarrow \Z^S$.
Associated to this, we defined in \cite{EGFS} the oriented, $S$-colored mod $p$ Albanese graph
$$
\Alb_p(\underline R)\coloneq \Alb_{p,1}(\underline R) ,
$$
as the Cayley graph of $W\coloneq \F_p^S/U_{\F_p}$ with respect to the generators given by the standard basis of $\F_p^S$.
 In other words, $\Alb_p(\underline R)$ has as vertices the set $W$, and the vertices $w$ and $w+\bar e_s$ are joined by a positively oriented edge of color $s$, where $e_s\in \F_p^S$ denotes the $s$-th basis vector.

\begin{definition} An {\rm $\F_p$-solution of $\underline R$ in $\Alb_p(\underline R)$} is a collection of $1$-chains $b_s\in C_1(\Alb_p(\underline R),\F_p)$ of color $s$, i.e.~a $\F_p$-linear combination of oriented edges of color $s$, such that the following holds.
If $c_s\in \F_p$ are such that $\sum_{s\in S} c_se_s\in U_{\F_{p}}\subset \F_p^S$, then the $1$-chain $\sum_{s\in S} c_sb_s$ is closed:
$\del \sum_{s\in S} c_sb_s=0 $.
\end{definition}

Since $\Alb_p(\underline R)$ is oriented, we can associate to each $1$-chain $b_s$ a number $\lambda_s(b_s)\in \F_p$, given by the sum of the coefficients of $b_s$, when written as a $\F_p$-linear combination of positively oriented edges of color $s$.
The color profile of $b=(b_s)_{s\in S}$ is then the vector $\lambda(b)=(\lambda_s(b_s))\in \F_p^S$.

The main result in \cite{EGFS} reduced Theorem \ref{thm:deg(C)} to the following purely combinatorial statement, which is the technical main result of this paper.

\begin{theorem} \label{thm:K_2p+1}
Let $p$ be a prime number and let $\underline R=M(K_{2p+1})$ be the graphic matroid associated to the complete graph $K_{2p+1}$ with edge set $S$.
Then any $\F_p$-solution $(b_s)_{s\in S}$ of $\underline R$ in the mod $p$ Albanese graph $\Alb_p(\underline R)$ of $\underline R$ has trivial color profile $\lambda((b_s)_{s\in S})=0\in \F_p^S$. 
\end{theorem}

In \cite[Definitions 8.3 and 8.9]{EGFS}, we defined the radical distance $d(\underline R)$ of a regular matroid $\underline R$ to the class of cographic matroids as the product of all primes $p$ such that any $\F_p$-solution of $\underline R$ in $\Alb_p(\underline R)$ with constant color profile has trivial color profile.
Theorem \ref{thm:K_2p+1} implies the following, which solves \cite[Problem 8.13]{EGFS} and partially answers \cite[Problem 8.14]{EGFS}.

\begin{corollary} \label{cor:d(R)}
The radical distance  $d(M(K_{n}))$ of the complete graph on $n$ vertices is divisible 
by all primes $p\leq (n-1)/2$.
In particular, $d(M(K_n))\geq 
e^{\vartheta(\frac{n-1}{2})}$ grows 
at least exponentially in $n$.
\end{corollary} 

The cases $p=2,3$ of Theorem \ref{thm:K_2p+1} were previously established in \cite{EGFS,EGFS-A6} using Sage and MAGMA code, respectively; any case for $p\geq 5$ is out of reach with that method, due to computational limitations.

In \cite{EGFS}, a similar computation was done at the prime $p=2$ for the complete bipartite graph $K_{3,3}$; the corresponding graphic matroid is a minor of $R_{10}$.
Using previous work of Gwena \cite{gwena} and Voisin \cite{voisin-JEMS}, this led to a proof of stable irrationality of very general cubic threefolds in \cite{EGFS}.
In an appendix to this paper, we generalize Theorem \ref{thm:K_2p+1} to the complete bipartite graph $K_{3,2p-1}$, where $p$ denotes an arbitrary prime; this yields in particular a conceptual proof of the aforementioned computation for $K_{3,3}$ and hence for $R_{10}$.

\begin{theorem} \label{thm:K_3,2p-1}
Let $p$ be a prime number and let $\underline R=M(K_{3,2p-1})$ be the graphic matroid associated to the complete bipartite graph $K_{3,2p-1}$ with edge set $S$.
Then any $\F_p$-solution $b=(b_s)_{s\in S}$ of $\underline R$ in the mod $p$ Albanese graph $\Alb_p(\underline R)$ has trivial color profile $\lambda(b)=0\in \F_p^S$. 
\end{theorem}

Let us explain the relation between the above theorems, which is given by the main result in \cite{EGFS}.
To this end, recall that to any regular matroid $\underline R$ of rank $g$ and ground set $S$, we can associate a projective family $\mathcal X^\star\to (\Delta^\star)^S$ of principally polarized abelian $g$-folds over an $|S|$-dimensional polydisc, whose monodromy is encoded by an integral realization $S\to U^\ast$, where $U= {\rm gr}^W_0H_1(X_t,\Z)$ for a base point $t$.
By \cite[Proposition 4.10]{EGFS-survey}, we may assume that this family extends to a matroidal family over some quasi-projective base $B^\star$, see \cite[Definition 4.1]{EGFS}.

\begin{theorem}[{\cite[Theorem 8.10]{EGFS}}] \label{thm:EGFS}
Let $\underline R$ be a regular matroid of rank $g$ on the ground set $S$ and let $\mathcal X^\star\to (\Delta^\star)^S$ be a family of $g$-dimensional principally polarized abelian varieties associated to $\underline R$ as above.
Assume that any $\F_p$-solution $b=(b_s)_{s\in S}$ of $\underline R$ in $\Alb_p(\underline R)$ with constant color profile $\lambda(b)=(c,c,\dots ,c)\in \F_p^S$ satisfies $\lambda(b)=0$, i.e.~$c=0$.
Then any curve $C\subset X_t$ on a very general fiber $X_t$ with class $[C]=m\cdot \theta_{g-1}$ satisfies $p\mid m$. 
\end{theorem}

\section{Linear relations of color profiles via test functions}

\subsection{The case of an arbitrary regular matroid}
Let $\underline R$ be a regular matroid on the ground set $S$ and with integral realization $S\to U^\ast$.
Consider the induced embedding $U\hookrightarrow \Z^S$ and consider its mod $p$ reduction: $U_{\F_p}\hookrightarrow \F_p^S$.
Let $W\coloneq \F_p^S/U_{\F_p}$ be the vertex set of $\Alb_p(\underline R)$.
A $1$-chain $b_s\in C_1(\Alb(\underline R),\F_p)$ of color $s\in S$ can be  written as
\begin{align} \label{eq:b_s}
b_s=\sum_{w\in W} f_s(w)[w\to w+e_s],
\end{align}
for some uniquely determined functions $f_s\colon W\to \F_p$, where $[w\to w+e_s]$ denotes the positively oriented edge of color $s$ that joins $w$ and $w+e_s$.
(Here and in what follows, we denote by slight abuse of notation the $s$-th basis vector $e_s\in\F_p^S$ and its image in the quotient $W=\F_p^S/U_{\F_p}$ by the same letter.)
Moreover, $\lambda_s(b_s)=\sum_{w\in W} f_s(w)$.

\begin{definition} \label{def:delta}
For a function $f\colon W\to \F_p$ and an element $e\in W$, we define the {\rm discrete derivative} $\Delta_e(f)\colon W\to \F_p$ by 
$$
\Delta_e(f)(w)\coloneqq f(w+e)-f(w).
$$
\end{definition}

The above operator is clearly linear in $f$. Moreover, we have the following Leibniz rule.

\begin{lemma} \label{lem:Leibniz-rule}
If $f,g\colon W\to \F_p$ are functions on $W$ and $e\in W$, then 
$$
\Delta_{e}(f\cdot g) =\Delta_{e}(f) \cdot g +f\cdot \Delta_{e}(g) +\Delta_{e}(f) \cdot \Delta_{e}(g)
$$
and
$$
\Delta_{e}(f\cdot g) =\Delta_{e}(f) \cdot g + (f\circ t_{e})\cdot \Delta_{e}(g) ,
$$
where $t_{e}\colon W\to W$, $x\mapsto x+e$, is translation by $e$.
\end{lemma}
\begin{proof}
For $w\in W$, we have
\begin{align*}
\Delta_{e}(f\cdot g)(w)&=f(w+e)g(w+e)-f(w)g(w) \\
&=f(w+e)(g(w+e)-g(w))+g(w)(f(w+e)-f(w)).
\end{align*}
This proves the second identity.
The other identity is proven similarly.
\end{proof}

\begin{lemma} \label{lem:Delta-solutions}
A collection of $1$-chains $(b_s)_{s\in S}$ as in \eqref{eq:b_s} is a solution of $\underline R$ in $\Alb_p(\underline R)$ if and only if for all $c_s$ with $\sum c_se_s\in U_{\F_p}$, we have
$$
\sum_{s \in S} c_s \Delta_{-e_s}(f_s)(w)=0\quad \quad \text{for all $w\in W$.}
$$
\end{lemma}
\begin{proof}
It suffices to prove that $\sum_{s \in S} c_sb_s$ is closed if and only if the function $\sum_{s \in S} c_s \Delta_{-e_s}(f_s)$ vanishes on $W$.
To see this, note that
\begin{align*}
\del \sum_{s \in S} c_sb_s &= \sum_{s \in S} \sum_{w\in W} c_sf_s(w) ([w+e_s]-[w]) ,
\end{align*}
where $[w]$ denotes the $0$-chain
on ${\rm Alb}_p(\underline{R})$ associated to the vertex 
$w$.
This vanishes if and only if for all $w\in W$, we have
$$
 \sum_{s \in S} c_s(f_s(w-e_s)-f_s(w))=0 .
$$
The claim in the lemma now follows.
\end{proof}

We aim to relate the vanishing of the function $\sum_{s \in S} c_s \Delta_{-e_s}(f_s)$ to the color profile $\sum_{w\in W} f_s(w)$.
The idea to achieve this is to consider the following pairing.

\begin{definition} \label{def:pairing}
For two functions $f,g\colon W\to \F_p$, we define the pairing
$$
\langle f,g\rangle\coloneq \sum_{w\in W}f(w)g(w)\in \F_p.
$$
\end{definition}

\begin{lemma} \label{lem:pairing-f,g}
For $e\in W$ and two functions $f,g\colon W\to \F_p$, we have
$$
\langle \Delta_e f,g\rangle=\langle f, \Delta_{-e} g\rangle .
$$
\end{lemma}
\begin{proof}
We have
\begin{align*}
\langle \Delta_e f,g\rangle& = \sum_{w\in W} (f(w+e)-f(w))g(w)\\
&= \sum_{w\in W}  f(w)g(w-e)-\sum_{w \in W} f(w)g(w)\\
&= \sum_{w\in W}  f(w)\Delta_{-e}g(w) .
\end{align*}
This proves the lemma.
\end{proof}

\begin{lemma} \label{lem:test-functions}
Assume that $(b_s)_{s\in S}$ as in \eqref{eq:b_s} is an $\F_p$-solution of $\underline R$ in $\Alb_p(\underline R)$, with $f_s\colon W\to \mathbb{F}_p$ the coefficient function for the $1$-chain $b_s$. 
Then for any (test) function $g\colon W\to \F_p$ and any $c_s\in \F_p$, $s\in S$, with $\sum_{s\in S}c_se_s\in U_{\F_p}$, 
we have
$$
\sum_{s\in S}  c_s \langle f_s ,\Delta_{e_s}g \rangle
=\sum_{s\in S} \sum_{w\in W} f_s(w) c_s \Delta_{e_s}g(w)= 0. 
$$
\end{lemma}
\begin{proof}
By Lemma \ref{lem:Delta-solutions}, $\sum_{s\in S} c_s \Delta_{-e_s}(f_s)$ vanishes on $W$.
Hence, for any function $g\colon W\to \F_p$, we have
$$
0=\langle \sum_{s\in S} c_s \Delta_{-e_s}(f_s) ,g \rangle  =\sum_{s\in S}  c_s  \langle  \Delta_{-e_s}(f_s) ,g \rangle .
$$
The claimed identity now follows from Lemma \ref{lem:pairing-f,g}.
\end{proof}

From the above lemma, we can deduce the following abstract criterion.

\begin{proposition} \label{prop:test-functions} 
Assume that there is an index set $I$, a family of test functions $g_i\colon W\to \F_p$, $i\in I$, and a family of elements $\sum_{s\in S} c_{si}e_s\in U_{\F_p}$, $i\in I$, such that the function
$$
\sum_{i\in I} c_{si} \Delta_{e_s}g_i\colon W\longrightarrow \F_p
$$
is constant, equal to $\mu_s\in F_p$. 
Then any $\F_p$-solution $b=(b_s)$ of $\underline R$ in $\Alb_p(\underline R)$ satisfies $\sum_{s\in S}\mu_s\lambda_s(b_s)=0$. 
\end{proposition}
\begin{proof}
By Lemma \ref{lem:test-functions}, 
$$
\sum_{s\in S} \sum_{w\in W} f_s(w) c_{si} \Delta_{e_s}g_i(w)= 0 \quad \quad \text{for all $i\in I$.}
$$
We take the sum over $i\in I$ and get
$$
0=\sum_{s\in S} \sum_{w\in W} f_s(w) \sum_{i\in I} c_{si} \Delta_{e_s}g_i(w)=\sum_{s\in S} \sum_{w\in W} f_s(w)\mu_s.
$$
This proves the proposition, because $\sum_{w\in W} f_s(w)=\lambda_s(b_s)$.
\end{proof}

\subsection{The case of graphic matroids}
Let us now specialize to the case where $\underline R=M(G)$ is the graphic matroid associated to a graph $G$ with edge set $S$ and vertex set $V$.
Then 
$$
U_{\F_p}=B^1(G,\F_p)\subset \F_p^S=C^1(G,\F_p)\quad \text{and} \quad W=C^1(G,\F_p)/B^1(G,\F_p)=H^1(G,\F_p).
$$
In particular, there is natural surjection
\begin{align*}
\delta\colon \F_p^{V}&\longrightarrow U_{\F_p}, \\
v&\mapsto \delta_v,
\end{align*}
where $\delta_v$ is the linear functional which maps an oriented edge $e\in C_1(G,\F_p)$ to
$$
\delta_v(e)=\begin{cases}
1\quad & \text{if $e$ points towards $v$};\\
-1 \quad & \text{if $e$ starts at $v$}; \\
0\quad &\text{otherwise}.
\end{cases}
$$
The following elements of $\F_p^S$ thus span the space $U_{\F_p}\subset \F_p^S$:  
$$
U_{\F_p}=\textrm{span}\left\{  \delta_v = \sum_{s\in S}\delta_v(s)e_s \mid v\in V\right\} .
$$ 
Applying Proposition \ref{prop:test-functions} to the index set $I=V$ and to the above elements, we get the following.

\begin{proposition}\label{prop:test-functions-1} 
Assume that there is a family of test functions $g_v\colon W\to \F_p$, $v\in V$, such that for every edge $s=[a\to b]$ of $G$, we have that
$$
\Delta_{e_s}(g_b-g_a) \colon W\longrightarrow \F_p
$$
is constant, equal to $\mu_s\in \F_p$.
Then any $\F_p$-solution $b=(b_s)$ of $\underline R$ in $\Alb_p(\underline R)$ satisfies $\sum_s\mu_s\lambda_s(b_s)=0$. 
\end{proposition}
\begin{proof}
By Proposition \ref{prop:test-functions} it suffices to show that
$$
\sum_{v\in V}  \delta_{v}(s) \Delta_{e_s}g_v(w)=\mu_s \quad \text{for all $w\in W$.}
$$
If $s=[a\to b]$, then $\delta_{v}(s)=1$ if $v=b$, $\delta_{v}(s)=-1$ if $v=a$, and $\delta_{v}(s)=0$ otherwise.
Hence,  
\begin{align*}
\sum_{v\in V}  \delta_{v}(s) \Delta_{e_s}g_v(w) &=\Delta_{e_s}(g_b-g_a)(w) .
\end{align*}
This is equal to $\mu_s\in \F_p$ for all $w\in W$ if and only if $\Delta_{e_s}(g_b-g_a)$ is a constant function, equal to $\mu_s$.
This concludes the proof. 
\end{proof}

We may apply this proposition
to restrict the possible color profiles
of $\mathbb{F}_p$-solutions of $M(G)$ in
${\rm Alb}_p(M(G))$, for sufficiently
connected graphs $G$:

\begin{proposition}\label{prop:linear-test-function}
Let $G$ be a biconnected
graph, i.e.~$G-v$ is connected for every
vertex $v\in V$.
Then there exists a family of linear test functions
$g_v\colon W\to \mathbb{F}_p$, $v\in V$, for which
$\Delta_{e_s}(g_b-g_a)$ is constant, and 
equal to $\mu_s\in \mathbb{F}_p$ for any tuple
$(\mu_s)_{s\in S}$ satisfying $\sum_{s\in S}\mu_s=0$.
\end{proposition}

\begin{proof} 
Since $W=H^1(G,\mathbb{F}_p)$,
a linear function $g_v\colon W\to \F_p$ is specified
by a cycle $\gamma_v\in H_1(G,\mathbb{F}_p)$.
Then $(g_b-g_a)(w) = w(\gamma_{b}-\gamma_a)$
is evaluation on $\gamma_b-\gamma_a$
and so $$\Delta_{e_s}(g_b-g_a) = (w+e_s)(\gamma_b-\gamma_a)-w(\gamma_b-\gamma_a) = e_s(\gamma_b-\gamma_a)$$
is the coefficient of the edge $s$ in the cycle
$\gamma_b-\gamma_a$.

Fix a vertex $v\in V$ and set $\gamma_v=\gamma\in H_1(G,\mathbb{F}_p)$
and $\gamma_{v'}=0$ for all $v'\neq v$. 
 We
may as well reorient the edges so that
all edges connecting to $v\in V$ are incoming edges.
Then, 
$\gamma_b-\gamma_a=0$ unless one of $a$ or $b=v$,
so suppose  $b=v$. 
Then $$\Delta_{e_s}(g_b-g_a) = 
\begin{cases} e_s(\gamma) & \textrm{if }[a\to b]\textrm{ is incoming to }b=v,\\
0 & \textrm{otherwise.}
\end{cases}$$
Let $[a\to v]$ and $[a'\to v]$ be two (incoming) edges
incident to $v$. By hypothesis, deleting the vertex
$G - v$ gives a connected graph,
so we may connect $a'$ to $a$ by a path $\gamma_0$
employing no edges adjacent to $v$. 
Then the concatenation
$\gamma\coloneqq 
[a\to v\to a'\to \gamma_0]\in 
H_1(G,\mathbb{F}_p)$
is a cycle satisfying 
$$e_s(\gamma) = 
\begin{cases} 1 & \textrm{if }
s=[a\to v],\\
-1 & \textrm{if }
s=[a'\to v],\\
0 & \textrm{otherwise.} \\
\end{cases}$$
Taking linear combinations of such
test functions, we deduce the proposition.
\end{proof}

\begin{corollary}
For a biconnected graph $G$,
any $\mathbb{F}_p$-solution of $\underline{R}=M(G)$
in ${\rm Alb}_p(\underline{R})$ has a constant
color profile $(\lambda_s(b_s))_{s\in S} = 
(\lambda, \dots, \lambda)$.
\end{corollary}

\begin{proof}
The result follows from Propositions 
\ref{prop:test-functions-1} and \ref{prop:linear-test-function}.
\end{proof}

\section{Construction of test functions for $K_{2p+1}$}

In this section we aim to construct test functions $g_v$, $v\in V$, for the complete graph $K_{2p+1}$ on $2p+1$ vertices, to which we can apply Proposition \ref{prop:test-functions-1}.
 
\subsection{Coordinates of $W$}

Let
$$
  G=K_{2p+1}, \qquad  V=V(G)=\{0,1,\ldots,2p\},
$$
and orient every edge from the smaller endpoint to the larger endpoint.
The edge set of $G$ is denoted as before by $S$.
For $0\leq i <j\leq 2p$, we denote the oriented edge which points from $i$ to $j$ by $[i\to j]\in S$.
Recall that 
$$
  W=H^1(G,\F_p)=C^1(G,\F_p)/B^1(G,\F_p).
$$

We choose the spanning tree of $G$ consisting of the edges $[0\to i]$, $1\le i\le 2p$ that emerge from $0$.
Therefore, the non-tree edges
$$
e_{ij}\coloneq [i\to j],  \qquad  1\le i<j\le 2p,
$$
give a basis of $W=H^1(G,\F_p)$.
It is the linear functional
on $H_1(G,\mathbb{F}_p)$ which assigns to a cycle
the coefficient of the edge $[i\to j]$ in it. 
We denote the dual basis of $W^\vee$ by
$$
x_{ij}\colon W\to \F_p, \qquad  1\le i<j\le 2p.
$$
It is the linear functional on $H^1(G,\mathbb{F}_p)$ which evaluates
a cocycle on the circuit of $G$ completed
by adding the edge $[i\to j]$ to the specified
spanning tree. That is, $x_{ij}$ is evaluation
on the triangular cycle $[0\to i\to j\to 0]\in H_1(G,\mathbb{F}_p)$. 
These functions yield coordinates on $W$. 
Note that $\Delta_{e_{ij}}(x_{kl})=\delta_{ij,kl}$ because the $x_{ij}$ are dual coordinates to the basis $e_{ij}$.   
For $1\le j<i\le 2p$ we write
$$
e_{ij}\coloneq -e_{ji}\quad \text{and} \quad x_{ij}=-x_{ji}.
$$
We further put $e_{ii}=0$ and $x_{ii}=0$.

\begin{lemma} \label{lem:x_0i}
The following relation holds in $W$: 
$$
e_{0i}=\sum_{j=1}^{2p}e_{ij}. 
$$
\end{lemma}
\begin{proof}
Consider the vertex $v=i$ and the associated coboundary $\delta_i\in B^1(G,\F_p)$.
Then 
$$
\delta_i=\sum_{j=0}^{i-1} e_{ji}-\sum_{j=i+1}^{2p} e_{ij}
$$ 
vanishes in $W$.
Since $ e_{ij}= -e_{ji}$ and $e_{ii}=0$, we get $e_{0i}=\sum_{j=1}^{2p}e_{ij}$, as claimed.
\end{proof}

\subsection{Existence of test functions for $K_{2p+1}$} 

We aim to prove the following.

\begin{theorem} \label{thm:existence-of-test-functions-K_2p+1}
Let $p$ be a prime number and let $G=K_{2p+1}$ be the complete graph on $2p+1$ vertices.
Then there are polynomial test functions $g_i\colon W=H^1(G,\F_p)\to \F_p$, $0\leq i\leq 2p$, which satisfy
$$
\Delta_{e_{ij}}(g_j-g_i)=\begin{cases}
1\quad &\text{if $(i,j)=(0,2p)$;}\\
0\quad &\text{if $0\leq i<j\leq 2p$ with $(i,j)\neq (0,2p)$.}
\end{cases}
$$
\end{theorem}

Since the edge $[0\to 2p]$, which plays a special role in the above argument, can be replaced up to relabelling by any other edge, the above theorem implies:

\begin{corollary} \label{cor:thm:existence-of-test-functions-K_2p+1}
In the above notation, the image of the map
\begin{align*} 
\operatorname{Maps}(W,\F_p)^V&\longrightarrow \operatorname{Maps}(W,\F_p)^S, \\
(g_i)_{0\leq i\leq 2p}&\mapsto (\Delta_{e_{ij}}(g_j-g_i))_{0\leq i<j\leq 2p}
\end{align*}
contains $\F_p^S$, where $\F_p\subset  \operatorname{Maps}(W,\F_p)$ denotes the subspace of constant functions $W\to \F_p$.
\end{corollary}

Theorem \ref{thm:K_2p+1} follows easily from Corollary \ref{cor:thm:existence-of-test-functions-K_2p+1} and Proposition \ref{prop:test-functions-1}.

\subsection{Proof of Theorem \ref{thm:existence-of-test-functions-K_2p+1}}
The remainder of this section is devoted to the proof of Theorem \ref{thm:existence-of-test-functions-K_2p+1}.
That is, we aim to construct polynomial test functions $g_i\colon W\to \F_p$ such that $\Delta_{e_{ij}}(g_j-g_i) $ is constant, equal to $1$ if $(i,j)=(0,2p)$ and equal to zero for all other indices $0\leq i<j\leq 2p$.
We start the construction by setting 
\begin{align} \label{def:g_2p}
g_{2p}\coloneq 0.
\end{align}

\begin{remark}
It is easy to see that linear functions $g_i$ will not work, so we are looking for polynomials of higher order.
Indeed, if $g_i$ is linear for all $i$, then $\Delta_{e_{ij}}(g_j-g_i) $ is automatically constant, equal to $\mu_{e_{ij}}\in \F_p$, say.
But $\sum_{0\leq i<j\leq 2p}  \mu_{e_{ij}}=0$, where the sum runs through all edges of $G$;
cf.~Proposition \ref{prop:linear-test-function}.
\end{remark}

The edge $e_{0,2p}$ plays a special role in Theorem \ref{thm:existence-of-test-functions-K_2p+1}.
By Lemma \ref{lem:x_0i},
$$
e_{0,2p}=-\sum_{i=1}^{2p-1}e_{i,2p}.
$$
Therefore, the coordinates $x_{i,2p} $ with $i=1,\dots ,2p-1$ play a special role and we introduce the notation
$$
z_i\coloneq x_{i,2p}\qquad \text{with $i\in \{1,\dots ,2p-1\}$}.
$$ 
Using these coordinates, Lemma \ref{lem:x_0i} together with the fact that $g_{2p}=0$ implies:
\begin{align} \label{eq:Delta_0,2p}
\Delta_{e_{0,2p}}(g_{2p}-g_0) 
&=g_0((x_{ij}) ,z_1,\dots ,z_{2p-1})-g_0((x_{ij}) ,z_1-1,z_2-1,\dots ,z_{2p-1}-1) ,
\end{align}
where $(x_{ij})$ stands for the remaining coordinate functions $(x_{ij})_{1\leq i<j\leq 2p-1}$.
We want to find a polynomial function $g_0$ such that the above difference is constant.

\begin{lemma} \label{lem:symmetric-functions}
Let $0\leq d\leq n$, and let $\sigma_d(z_1,\dots,z_{n})\in \F_p[z_1,\dots,z_{n}]$ be the elementary symmetric polynomial of degree $d$ over $\F_p$.
Then
\begin{align} \label{eq:lem:symmetric-functions}
\sigma_d(z_1+1,\dots,z_{n}+1)-\sigma_d(z_1,\dots,z_{n})=\sum_{j=0}^{d-1}\binom{n-j}{d-j}\sigma_j(z_1,\dots,z_{n}).
\end{align}
Here we use the convention $\sigma_0=1$.
In particular, for $d=p$ and $n=2p-1$, one has
$$
\sigma_p(z_1+1,\dots,z_{2p-1}+1)-\sigma_p(z_1,\dots,z_{2p-1})=1.
$$
\end{lemma}

\begin{proof}
Write $\sigma_j=\sigma_j(z_1,\dots,z_{n})$.
We use the generating function for elementary symmetric polynomials:
$$
\prod_{i=1}^{n}(1+z_i t)=\sum_{j=0}^{n}\sigma_j t^j.
$$
After replacing each $z_i$ by $z_i+1$, we get
$$
\prod_{i=1}^{n}(1+(z_i+1)t)=\prod_{i=1}^{n}(1+t+z_i t).
$$
Factoring $1+t$ from each term gives
\begin{align*}
\prod_{i=1}^{n}(1+(z_i+1)t)
&=(1+t)^{n}\prod_{i=1}^{n}\left(1+\frac{z_i t}{1+t}\right) \\
&=(1+t)^{n}\sum_{j=0}^{n}\sigma_j\left(\frac{t}{1+t}\right)^j \\
&= \sum_{j=0}^{n}\sigma_j t^j(1+t)^{n-j}.
\end{align*} 
Comparing coefficients of $t^d$, we obtain 
$$
\sigma_d(z_1+1,\dots,z_{n}+1)=\sigma_d+\sum_{j=0}^{d-1}\binom{n-j}{d-j}\sigma_j.
$$
This proves \eqref{eq:lem:symmetric-functions}.

To prove the simplified formula  when $d=p$ and $n=2p-1$, it then suffices to show the following identity in $\F_p$: 
$$
\binom{2p-1-j}{p-j}=\begin{cases}
1,&\quad \text{for }j=0,\\
0,&\quad \text{for }1\leq j\leq p-1.
\end{cases}
$$
The case $1\leq j\leq p-1$ is clear, because in that case $p\leq 2p-1-j$ but $p>p-j$ and $p>2p-1-j-(p-j)=p-1$.
Finally, for $j=0$ we have
\begin{align*}
\binom{2p-1}{p} 
 =\frac{(2p-1)(2p-2)\cdots (p+1) }{(p-1)!} 
 =\frac{(p-1)(p-2)\cdots 1 }{(p-1)!}=1 \in \F_p.
\end{align*}  
This proves the lemma.
\end{proof}

Motivated by the above lemma, we set
\begin{align} \label{def:g_0}
g_{0}\coloneq \sigma_p(z_1,\dots,z_{2p-1}).
\end{align}
By \eqref{eq:Delta_0,2p} and Lemma \ref{lem:symmetric-functions}, we then have
$$
\Delta_{e_{0,2p}}(g_{2p}-g_0)=1.
$$

\begin{lemma} \label{lem:Delta-g_i-necessary}
Let $g_i\colon W\to \F_p$ be polynomial functions for $i=0,\dots ,2p$.
Assume that $g_{2p}=0$ and $g_0$ as in \eqref{def:g_0}.
Then  $\Delta_{e_{ij}}(g_j-g_i)$ satisfy the conditions from Theorem \ref{thm:existence-of-test-functions-K_2p+1} if and only if the following holds: 
\begin{align}
\Delta_{e_{0i}}(g_i) &=\sigma_{p-1}(z_1,\dots,\widehat{z_i},\dots,z_{2p-1})\quad &\text{for $1\leq i\leq 2p-1$}; \label{item:lem:Delta-g_i-necessary:1}\\ 
\Delta_{e_{ij}}(g_i) &= \Delta_{e_{ij}}(g_j) \quad &\text{for $1\leq i< j\leq 2p$}.\label{item:lem:Delta-g_i-necessary:2}
\end{align}
\end{lemma}

\begin{proof}
Condition \eqref{item:lem:Delta-g_i-necessary:2} is equivalent to $\Delta_{e_{ij}}(g_j-g_i)=0$ for $1\leq i<j\leq 2p$.
By Lemma \ref{lem:symmetric-functions}, it thus suffices to show that \eqref{item:lem:Delta-g_i-necessary:1} is equivalent to 
$
\Delta_{e_{0i}}(g_i-g_{0})=0
$
for $1\leq i\leq 2p-1$.
The latter is equivalent to 
\begin{align*}
\Delta_{e_{0i}}(g_i)=\Delta_{e_{0i}}(g_0)
&=\Delta_{\sum_{j=1}^{2p} e_{ij}}(g_0)\\
&=\sigma_p(z_1,\dots, z_{i-1},z_i+1,z_{i+1},\dots, z_{2p-1})-\sigma_p(z_1, \dots, z_{2p-1}) ,
\end{align*}
where the first equality follows from Lemma \ref{lem:x_0i}.
We may then write
$$
\sigma_p(z_1, \dots, z_{2p-1})=z_i\sigma_{p-1}(z_1, \dots,\widehat z_i,\dots, z_{2p-1})+\sigma_p(z_1, \dots,\widehat z_i,\dots z_{2p-1}).
$$
Using this, we find
\begin{align}\label{eq:Delta-sigma_p=hat-sigma-p-1}
\sigma_p(z_1,\dots, z_{i-1},z_i+1,z_{i+1},\dots, z_{2p-1})-\sigma_p(z_1, \dots, z_{2p-1}) =
 \sigma_{p-1}(z_1, \dots,\widehat z_i,\dots, z_{2p-1}).
\end{align}
Altogether, the condition $
\Delta_{e_{0i}}(g_{i}-g_0)=0
$
for $1\leq i\leq 2p-1$ is thus equivalent to \eqref{item:lem:Delta-g_i-necessary:1}, as we want.
\end{proof}

For $1\leq i\leq 2p-1$, we now try the following ``Ansatz'':
\begin{align} \label{def:g_i}
g_i\coloneq \sum_{j=1}^{2p-1} x_{ij}c_{ij} ,
\end{align}
where $c_{ij}\in \F_p[z_1,\dots ,z_{2p-1}]$. As before we have $x_{ij}=-x_{ji}$ and  $x_{ii}=0$; we impose accordingly that $c_{ii}=0$.

\begin{lemma} \label{lem:c_ij}
Consider the test functions $g_i\colon W\to \F_p$ given by  \eqref{def:g_2p}, \eqref{def:g_0},  and  \eqref{def:g_i}.
For $e\in W$, let $t_e\colon W\to W$ be the translation by $e$.
Then
\begin{align}
\Delta_{e_{ij}}(g_i) &= c_{ij}\quad \text{and}\quad \Delta_{e_{ij}}(g_j)  = -c_{ji}  \quad &\text{for $1\leq i< j\leq 2p-1$};\label{item:lem:c_ij:2}\\
\Delta_{e_{0i}}(g_i) &=\sum_{j=1}^{2p-1} c_{ij}\circ t_{e_{i,2p}}+x_{ij}\Delta_{e_{i,2p}}c_{ij} \quad &\text{for $1\leq i\leq 2p-1$}; \label{item:lem:c_ij:1}\\ 
\Delta_{e_{i,2p}}(g_i) &= \sum_{l=1}^{2p-1}x_{il}\Delta_{e_{i,2p}}c_{il} \quad &\text{for $1\leq i\leq 2p-1$}.\label{item:lem:c_ij:4} 
\end{align} 
\end{lemma}
\begin{proof}
In all cases, $1\leq i\leq 2p-1$, and so $g_i$ is given by \eqref{def:g_i}.

Let us first prove \eqref{item:lem:c_ij:2}. 
To this end, let  $1\leq i< j\leq 2p-1$.
Using \eqref{def:g_i}, we find
$$
\Delta_{e_{ij}}(g_i) = \sum_{l=1}^{2p-1} \Delta_{e_{ij}}(x_{il}c_{il})=c_{ij} ,
$$
as we want.
Similarly, using $x_{ij}=-x_{ji}$, we get
$$
\Delta_{e_{ij}}(g_j) =  \sum_{l=1}^{2p-1} \Delta_{e_{ij}}(x_{jl}c_{jl})=- \sum_{l=1}^{2p-1} \Delta_{e_{ij}}(x_{lj}c_{jl})=-c_{ji} .
$$
This proves \eqref{item:lem:c_ij:2}. 

Next, let $1\leq i\leq 2p-1$.
We use Lemma \ref{lem:x_0i} to write $e_{0i}=\sum_{j=1}^{2p} e_{ij}$.
Using this and the Leibniz rule from Lemma \ref{lem:Leibniz-rule}, we find
$$
\Delta_{e_{0i}}(g_i) = \sum_{j=1}^{2p-1} \Delta_{\sum_{l=1}^{2p} e_{il}} (c_{ij}x_{ij})=\sum_{j=1}^{2p-1} c_{ij}\circ t_{e_{i,2p}}+x_{ij}\Delta_{e_{i,2p}}c_{ij} .
$$ 
This concludes the proof of \eqref{item:lem:c_ij:1}.
The proof of \eqref{item:lem:c_ij:4} is similar. 
This concludes the proof of the lemma.
\end{proof}

\begin{lemma} \label{lem:c_ij-2}
Assume that there are polynomials $c_{ij}\in \F_p[z_1,\dots ,z_{2p-1}]$ for $1\leq i,j\leq 2p-1$ with the following properties:
\begin{enumerate}
\item $c_{ij}=-c_{ji}$;\label{item:lem:c_ij-2:1}
\item $c_{ij}$ is independent of $z_i$  
(and $z_j$);\label{item:lem:c_ij-2:2}
\item $\sum_{j=1}^{2p-1} c_{ij}= \sigma_{p-1}(z_1,\dots,\widehat{z_i},\dots,z_{2p-1}) $.\label{item:lem:c_ij-2:3}
\end{enumerate}
Then the test functions $g_i\colon W\to \F_p$ given by \eqref{def:g_2p}, \eqref{def:g_0},   and  \eqref{def:g_i} satisfy the conclusion of Theorem \ref{thm:existence-of-test-functions-K_2p+1}.
\end{lemma}
\begin{proof}
This is a direct consequence of Lemmas \ref{lem:Delta-g_i-necessary} and \ref{lem:c_ij}. 
Indeed, condition \eqref{item:lem:c_ij-2:2} ensures $\Delta_{e_{i,2p}}c_{ij} =0$ and $ c_{ij}\circ t_{e_{i,2p}}= c_{ij} $.
Hence, equation \eqref{item:lem:c_ij:1} in Lemma \ref{lem:c_ij} yields
$$
\Delta_{e_{0i}}(g_i) =\sum_{j=1}^{2p-1} c_{ij}
$$
which agrees with $\sigma_{p-1}(z_1,\dots,\widehat{z_i},\dots,z_{2p-1})$ by condition \eqref{item:lem:c_ij-2:3}.
Thus, equation \eqref{item:lem:Delta-g_i-necessary:1} in 
Lemma \ref{lem:Delta-g_i-necessary} holds.

Similarly, condition \eqref{item:lem:c_ij-2:1} implies by  \eqref{item:lem:c_ij:2} in Lemma \ref{lem:c_ij} that $\Delta_{e_{ij}}(g_i)=\Delta_{e_{ij}}(g_j)$ and so equation \eqref{item:lem:Delta-g_i-necessary:2} in Lemma \ref{lem:Delta-g_i-necessary} holds true for $j<2p$.
The same holds for $j=2p$, because
$$
\Delta_{e_{i,2p}}(g_i) = \sum_{l=1}^{2p-1}x_{il}\Delta_{e_{i,2p}}(c_{il}) =0
$$
for $1\leq i\leq 2p-1$, since $\Delta_{e_{i,2p}}(c_{il})=0$ as $c_{il}$ is independent of $z_i$. 
By Lemma \ref{lem:Delta-g_i-necessary}, 
the lemma follows.
\end{proof}

The final step in the proof of Theorem \ref{thm:existence-of-test-functions-K_2p+1} is the construction of polynomials  $c_{ij}\in \F_p[z_1,\dots ,z_{2p-1}]$ which satisfy conditions \eqref{item:lem:c_ij-2:1}--\eqref{item:lem:c_ij-2:3} in Lemma \ref{lem:c_ij-2}.

\begin{lemma}\label{lem:existence-cij}
There are polynomials $c_{ij}\in \F_p[z_1,\dots ,z_{2p-1}]$ for $1\leq i,j\leq 2p-1$ that satisfy conditions \eqref{item:lem:c_ij-2:1}--\eqref{item:lem:c_ij-2:3} in Lemma \ref{lem:c_ij-2}.
\end{lemma}

\begin{proof}
Let $I\coloneq  \{1,\dots ,2p-1\}$.
For a subset $J\subset I$, we write
$$
z_J\coloneq \prod_{j\in J}z_j
$$
and make the ``Ansatz''
$$
c_{ij}\coloneq \sum_{\substack{J\subset I\setminus\{i,j\}\\ |J|=p-1}} a_{i,j,J}\cdot z_J
$$
with $a_{i,j,J}\in \F_p$ and $a_{i,i,J}=0$.
By construction, this function is independent of $z_i$ and $z_j$.
Moreover, the antisymmetry $c_{ij}=-c_{ji}$ is equivalent to 
\begin{align}\label{item:a_ij:1}
a_{i,j,J}=-a_{j,i,J} .
\end{align}
Finally, condition \eqref{item:lem:c_ij-2:3} in Lemma \ref{lem:c_ij-2} translates for $1\leq i\leq 2p-1$ into the condition
\begin{align*} 
\sum_{j=1}^{2p-1}  \sum_{\substack{J\subset I\setminus\{i,j\}\\ |J|=p-1}} a_{i,j,J} z_J= \sum_{\substack{J\subset I\setminus \{i\}\\ |J|=p-1}} z_J .
\end{align*}
Comparing coefficients of $z_J$, we see that the above is equivalent to the following condition for any subset $J\subset I\setminus\{i\}$ of cardinality $p-1$:
\begin{align}\label{item:a_ij:2}
\sum_{j\in I\setminus (J\cup \{i\})}  a_{i,j,J} =1 .
\end{align}
For any subset  $J\subset I\setminus\{i\}$ of cardinality $p-1$, a matrix of coefficients $a_{i,j,J}$ with $i,j\in I\setminus J$ (which can be chosen uniformly for each $J$)
is constructed in Lemma \ref{lem:existence:a_i,j,J}.
This concludes the proof.
\end{proof}

\begin{lemma} \label{lem:existence:a_i,j,J}
There is a skew-symmetric $p\times p$ matrix $A=(a_{ij})\in \F_p^{p\times p}$ such that for all $1\leq i\leq p$, we have
$$
\sum_{j=1}^p a_{ij}=1\in \F_p.
$$
\end{lemma}
\begin{proof}
It is enough to take the following matrix:
$$
A=\begin{pmatrix}0&-1&-1&\cdots&-1\\1&0&0&\cdots&0\\1&0&0&\cdots&0\\\vdots&\vdots&\vdots&\ddots&\vdots\\1&0&0&\cdots&0\end{pmatrix} \in \F_p^{p\times p}.
$$
This matrix satisfies the required conditions because $-(p-1)=1\in \F_p$.
\end{proof}

\begin{example}[Explicit test functions for $p=2$]
In the special case $p=2$, the above argument yields the following test functions, where as before,
$z_1=x_{14}$, $z_2=x_{24}$, and $z_3=x_{34}$:
\[
\begin{aligned}
g_0&=z_1z_2+z_1z_3+z_2z_3,\\
g_1&=x_{12}z_3+x_{13}z_2,\\
g_2&=x_{12}z_3+x_{23}z_1,\\
g_3&=x_{13}z_2+x_{23}z_1,\\
g_4&=0.
\end{aligned}
\] 
\end{example}

\begin{example}[Explicit test functions for $p=3$]
Let now $p=3$.
Then the above argument yields the following test functions:
\[
\begin{aligned}
g_0
&=\sigma_3(z_1,z_2,z_3,z_4,z_5)\\
&=
 z_1z_2z_3+z_1z_2z_4+z_1z_2z_5
+z_1z_3z_4+z_1z_3z_5+z_1z_4z_5
+z_2z_3z_4+z_2z_3z_5+z_2z_4z_5
+z_3z_4z_5.
\end{aligned}
\]
%
%
%
%
Define \(c_{ji}=-c_{ij}\), \(c_{ii}=0\), and for \(1\le i<j\le 5\) set
\[
\begin{array}{lll}
c_{12}=-(z_3z_4+z_3z_5+z_4z_5),
&
c_{13}=-(z_2z_4+z_2z_5+z_4z_5),
&
c_{14}=-(z_2z_3+z_2z_5+z_3z_5),
\\[3pt]
c_{15}=-(z_2z_3+z_2z_4+z_3z_4),
&
c_{23}=-(z_1z_4+z_1z_5),
&
c_{24}=-(z_1z_3+z_1z_5),
\\[3pt]
c_{25}=-(z_1z_3+z_1z_4),
&
c_{34}=-z_1z_2,
&
c_{35}=-z_1z_2,
\\[3pt]
c_{45}=0.
\end{array}
\]
Then
\[
g_i=\sum_{j=1}^5 x_{ij}c_{ij}
\qquad (1\le i\le 5).
\]
Explicitly, this gives
\[
\begin{aligned}
g_1={}&
-x_{12}(z_3z_4+z_3z_5+z_4z_5)
-x_{13}(z_2z_4+z_2z_5+z_4z_5)\\
&-x_{14}(z_2z_3+z_2z_5+z_3z_5)
-x_{15}(z_2z_3+z_2z_4+z_3z_4),
\\[4pt]
g_2={}&
-x_{12}(z_3z_4+z_3z_5+z_4z_5)
-x_{23}(z_1z_4+z_1z_5)\\
&-x_{24}(z_1z_3+z_1z_5)
-x_{25}(z_1z_3+z_1z_4),
\\[4pt]
g_3={}&
-x_{13}(z_2z_4+z_2z_5+z_4z_5)
-x_{23}(z_1z_4+z_1z_5)\\
&-x_{34}z_1z_2
-x_{35}z_1z_2,
\\[4pt]
g_4={}&
-x_{14}(z_2z_3+z_2z_5+z_3z_5)
-x_{24}(z_1z_3+z_1z_5)
-x_{34}z_1z_2,
\\[4pt]
g_5={}&
-x_{15}(z_2z_3+z_2z_4+z_3z_4)
-x_{25}(z_1z_3+z_1z_4)
-x_{35}z_1z_2,
\\[4pt]
g_6={}&0.
\end{aligned}
\] 
\end{example}

\section{Proof of main results}

\begin{proof}[Proof of Theorem \ref{thm:K_2p+1}]
This follows easily from Corollary \ref{cor:thm:existence-of-test-functions-K_2p+1} and Proposition \ref{prop:test-functions-1}.
\end{proof}

\begin{proof}[Proof of Corollary \ref{cor:d(R)}]
By Theorem \ref{thm:K_2p+1}, the color profile of any $\F_p$-solution of $M(K_{2p+1})$ in its own mod $p$ Albanese graph has trivial color profile.
By \cite[Proposition 7.2]{EGFS}, the same holds for any regular matroid that contains the graphic matroid $M(K_{2p+1})$ as a minor,
hence applies to $M(K_n)$ for all $n\geq 2p+1$.
So $p\mid d(M(K_n))$ for all primes 
$p$ with $p\leq (n-1)/2$. 
This concludes the proof.
\end{proof}

\begin{proof}[Proof of Theorem \ref{thm:deg(C)}] 
Let $k<g$ be positive integers and let $p\leq (k+1)/2$ be a prime.
Let $Z\subset X$ be a codimension $k$ subvariety of a very general principally polarized abelian variety $(X,\Theta_X)$ of dimension $g$.
Since $(X,\Theta_X)$ is very general, its Mumford--Tate group is maximal and 
the only Hodge classes in $H^{2k}(X,\Z)$ are multiples of the minimal class $\theta_k$.
Hence, $[Z]=m\cdot \theta_k$ and we aim to show that $p\mid m$.
Note that showing this for all subvarieties $Z\subset X$ of codimension $k$ is equivalent to showing it for all cycles in the Chow group $\CH^k(X)$ of codimension $k$-cycles, and we will use this in our inductive argument below. 
\smallskip

\textbf{Step 1.} We first deal with the special case where $g=2p$ and $k=g-1$, i.e.~$Z$ is a curve and $(X,\Theta_X)$ has dimension $2p$.

 We consider the matroidal family of principally polarized abelian varieties of dimension $2p$ that is associated to the graphic matroid $M(K_{2p+1})$ of the complete graph $K_{2p+1}$.
 A spread out argument shows that the very general fiber of this family carries a curve of class $m\cdot \theta_{2p-1}$.
Theorem \ref{thm:K_2p+1} together with the main result in \cite{EGFS} (cf.~Theorem \ref{thm:EGFS}) then implies $p\mid m$, as we want.  \smallskip

\textbf{Step 2.} We deal with the case where $k=g-1$ and $g$ is arbitrary (with $p\leq g/2$).

We argue by induction on $g-2p$.
If $g-2p=0$, then the result follows from Step 1.
If $g-2p\geq 1$, then we specialize $(X,\Theta_X)$ to the product $(Y,\Theta_Y)\times (E,\Theta_E)$ where $(Y,\Theta_Y)\in \mathcal A_{g-1}$ is very general.
Let $f\colon X\to Y$ be the projection.
Then
$$
[f_\ast Z]=m\cdot f_\ast [\Theta_X]^{g-1}/(g-1)!=m[\Theta_Y]^{g-2}/(g-2)! .
$$
Applying the induction hypothesis to the cycle $f_\ast Z$ on $Y$, we get that $p\mid m$, as we want.   \smallskip

\textbf{Step 3.} We finally deal with the general case where $k<g$ and $p\leq (k+1)/2$ is a prime.

We argue by induction on $g-k$.
If $g-k=1$, then $k=g-1$ and we conclude via Step 2.
If $g-k>1$, then we specialize $(X,\Theta)$ to a product $(Y,\Theta_Y)\times (E,\Theta_E)$ with $(Y,\Theta_Y)\in \mathcal A_{g-1}$ very general. 
Let $\iota \colon Y\to X$ be the inclusion of $Y$ as an abelian subvariety.
Then
$$
[\iota^\ast Z]=m\cdot \iota^\ast [\Theta_X]^k/k!=m\cdot [\Theta_Y]^k/k! .
$$
Applying the induction hypothesis to the cycle $\iota^\ast Z$ on $Y$, we then find that $p\mid m$, as we want. 
This concludes the proof of the theorem.
\end{proof}

\begin{proof}[Proof of Corollary \ref{cor:thm:deg(C)}]
The conclusion $\prod_{p\,\leq\, (k+1)/2}p =
  e^{\vartheta\left(\frac{k+1}{2}\right)} \mid \mkg $ follows directly from Theorem \ref{thm:deg(C)}, giving the corollary.
\end{proof}

The following estimate follows from work of Dusart \cite{Dusart}.

\begin{corollary}\label{cor:Dusart}  We have
$$
\mkg >\exp\!\left( \frac {k+1}{2} \left(1-\frac{151.3}{\left(\log\left(\frac{k+1}{2}\right)\right)^4}\right) \right) \qquad \text{for all $3\leq k<g$} 
$$
and
$$
\mkg > e^{(k+1)/3} \qquad \text{for all \ $33\leq k <g $}.
$$
\end{corollary}
\begin{proof} 
By Corollary \ref{cor:thm:deg(C)},
$$
\prod_{p\leq (k+1)/2}p \mid \mkg . 
$$
Hence
\begin{align}\label{log-bound}
\log  \mkg  
\geq \sum_{p\leq (k+1)/2}\log p=\vartheta\!\left(\frac{k+1}{2}\right).
\end{align}
By Dusart's explicit estimate (see \cite[Theorem~4.2]{Dusart}, using the row with exponent $4$),
$$
\left|\vartheta(x)-x\right|<\frac{151.3x}{(\log x)^4} \qquad \text{for $x\geq 2$},
$$  
Substituting $x=(k+1)/2$, plugging into (\ref{log-bound}), and exponentiating proves the first inequality.

To prove the second estimate, apply \cite[Theorem~4.2]{Dusart} (using the row with exponent $2$), to get 
$$
\left|\vartheta(x)-x\right| < \frac{3.965 x}{(\log x)^2} \qquad \text{for all }\ x\geq 2.
$$
Thus
$$
\vartheta(x)>
x\left(1-\frac{3.965}{(\log x)^2}\right).
$$
For $k\geq 62$, we have $x\geq 31.5$, and hence
$$
\frac{3.965}{(\log x)^2} \leq \frac{3.965}{(\log 31.5)^2} \approx 0.333126 < \frac{1}{3}.
$$
Consequently,
$$
\log \left( \mkg  \right) \geq \vartheta(x) > \frac{2x}{3} = \frac{k+1}{3}.
$$
Exponentiating proves $ \mkg  >e^{(k+1)/3}$ for $k\geq 62$.
To bring this down to $k\geq 33$, we use an explicit computation in the range $33\leq k\leq 61$ that is illustrated in the following table:
\[
\begin{array}{c|c|c}
k & \displaystyle\prod_{p\leq (k+1)/2}p
  & \text{upper bound for } e^{(k+1)/3} \\ \hline
33\leq k\leq 36
  & 510\,510
  & e^{37/3}\approx 227\,143 \\[2pt]
37\leq k\leq 44
  & 9\,699\,690
  & e^{15}\approx 3\,269\,017 \\[2pt]
45\leq k\leq 56
  & 223\,092\,870
  & e^{19}\approx 178\,482\,301 \\[2pt]
57\leq k\leq 60
  & 6\,469\,693\,230
  & e^{61/3}\approx 677\,102\,575 \\[2pt]
k=61
  & 200\,560\,490\,130
  & e^{62/3}\approx 944\,972\,767 .
\end{array}
\]
This concludes the proof of the second inequality in the corollary. 
\end{proof}

\begin{remark}
We cannot extend the inequality $e^{(k+1)/3}\leq e^{\vartheta\left(\frac{k+1}{2}\right)}$ below $k=33$.
Indeed, for $k=32$, the relevant product of primes is only $2\cdot 3\cdot 5\cdot 7\cdot 11\cdot 13=30030$, whereas $e^{11}\approx 59874>30030$.
\end{remark}

\begin{proof}[Proof of Theorem \ref{thm:K_3,2p-1}]
This follows from Proposition \ref{prop:test-functions-1} applied to the test functions for the complete bipartite graph $K_{3,2p-1}$ that we construct in Theorem \ref{thm:existence-of-test-functions-K_3,2p-1} below (cf.~Corollary \ref{cor:thm:existence-of-test-functions-K_3,2p-1}).
\end{proof}

\appendix

\section{Construction of test functions for $K_{3,2p-1}$}
Let
$$
  G=K_{3,2p-1}, \qquad  V=V(G)=\{L_0,L_1,L_2\}\cup \{R_0,R_1,\ldots,R_{2p-2}\},
$$
be a bipartite graph with edge set 
$$
S\coloneq \{e_{ij}=[L_i\to R_j]\mid 0\leq i\leq 2,\ 0\leq j\leq 2p-2\},
$$
where we orient all edges from left to right. 
We choose the spanning tree
$$
T=\{e_{0i}\mid 0\leq i\leq 2p-2\}\cup \{e_{10},e_{20}\} .
$$
We identify (oriented) edges with their images in $H^1(G,\F_p)$. 
A basis of $W\coloneq H^1(G,\F_p)$ is given by
$$
e_{1i},e_{2i}\qquad 1\leq i\leq 2p-2 .
$$

\begin{lemma} \label{lem:e_ij-2}
The following relations hold in $W$:
\begin{align*}
e_{0i}=-e_{1i}-e_{2i}, \qquad 
e_{10}= -\sum_{i=1}^{2p-2}e_{1i}, \qquad 
e_{20}= -\sum_{i=1}^{2p-2}e_{2i} .
\end{align*}
\end{lemma}
\begin{proof}
This follows from a simple coboundary computation.
\end{proof}

Motivated by the above lemma, we denote the coordinates of $W$ which correspond to the basis $e_{1i},e_{2i}$, $1\leq i\leq 2p-2 $, by $x_i,y_i$ with $1\leq i\leq 2p-2$.
Here $x_i$ is dual to $e_{1i}$ while $y_i$ is dual to $e_{2i}$. 
By the above lemma, we then have the following.

\begin{equation}
\label{eq:Delta-actions}
\begin{aligned}
\Delta_{e_{10}}:\quad
&\begin{cases} x_i \mapsto -1,\\ y_i \mapsto 0, \end{cases}
&& 1\leq i\leq 2p-2, \\[1em]
\Delta_{e_{20}}:\quad
&\begin{cases} x_i \mapsto 0,\\ y_i \mapsto -1, \end{cases}
&& 1\leq i\leq 2p-2, \\[1em]
\Delta_{e_{00}}:\quad
&\begin{cases} x_i \mapsto 1,\\ y_i \mapsto 1, \end{cases}
&& 1\leq i\leq 2p-2, \\[1em]
\Delta_{e_{0i}}:\quad
&\begin{cases} x_j \mapsto -1, & j=i,\\ x_j \mapsto 0,   & j\neq i,\\ y_j \mapsto -1, & j=i,\\ y_j \mapsto 0,   & j\neq i, \end{cases}
&& 1\leq i,j\leq 2p-2.
\end{aligned}
\end{equation}
Note that \eqref{eq:Delta-actions} together with the Leibniz rule from Lemma \ref{lem:Leibniz-rule} completely determines the action of $\Delta_{e_{ij}}$ on polynomial functions on $W$.

\subsection{Existence of test functions for $K_{3,2p-1}$} 

We aim to prove the following.

\begin{theorem} \label{thm:existence-of-test-functions-K_3,2p-1}
Let $p$ be a prime number and let $G=K_{3,2p-1}$ be the complete bipartite graph with parts of cardinalities $3$ and $2p-1$.
Then there are polynomial test functions $g_{L_i},g_{R_j}\colon W=H^1(G,\F_p)\to \F_p$ which satisfy
$$
\Delta_{e_{ij}}(g_{R_j}-g_{L_i})=\begin{cases}
1\quad &\text{if $(i,j)=(0,0)$;}\\
0\quad &\text{if $0\leq i \leq 2$ and $0\leq j\leq 2p-2$ with $(i,j)\neq (0,0)$.}
\end{cases}
$$
\end{theorem}

Since the edge $e_{00}=[L_0\to R_0]$, which plays a special role in the above argument, can be replaced up to relabelling by any other edge, the above theorem implies:

\begin{corollary} \label{cor:thm:existence-of-test-functions-K_3,2p-1}
In the above notation, the image of the map
$$
\operatorname{Maps}(W,\F_p)^V\longrightarrow \operatorname{Maps}(W,\F_p)^S,\qquad ((g_{L_i})_{0\leq i\leq 2}, (g_{R_j})_{0\leq j\leq 2p-2}) \mapsto (\Delta_{e_{ij}}(g_{R_j}-g_{L_i}))_{0\leq i\leq 2,\ 0\leq j\leq 2p-2}
$$
contains $\F_p^S$, where $\F_p\subset  \operatorname{Maps}(W,\F_p)$ denotes the constant functions $W\to \F_p$.
\end{corollary}


\subsection{Proof of Theorem \ref{thm:existence-of-test-functions-K_3,2p-1}}
The remainder of this section is devoted to the proof of Theorem \ref{thm:existence-of-test-functions-K_3,2p-1}.


The following consequence of Lemma \ref{lem:symmetric-functions} will be useful.

\begin{lemma} \label{lem:symmetric-functions-2}
The following identities hold in the polynomial ring $\F_p[y_1,\ldots,y_{2p-2}]$:
$$
\sigma_{p-1}(y_1+1,\ldots,y_{2p-2}+1)=\sigma_{p-1}(y_1,\ldots,y_{2p-2}),
$$
and
$$
\sigma_p(y_1+1,\ldots,y_{2p-2}+1)-\sigma_p(y_1,\ldots,y_{2p-2})=1-\sigma_{p-1}(y_1,\ldots,y_{2p-2}).
$$
\end{lemma}
\begin{proof}
This follows directly from Lemma \ref{lem:symmetric-functions} and the fact that the binomial coefficient
$$
\binom{2p-2-j}{d-j}
$$
has the following divisibility properties:

If $d=p-1$, then it is divisible by $p$ for all $j=0,\dots, p-2$.

If $d=p$,  then it is divisible by $p$ for all $j=1,\dots, p-2$.
Moreover,
$$
\binom{2p-2-0}{p-0}\equiv \frac{(2p-2)(2p-3)\dots (p+1)}{(p-2)!}\equiv 1 \mod p
$$
and
$$
\binom{2p-2-(p-1)}{p-(p-1)}=\binom{p-1}{1} \equiv -1 \mod p.
$$
This concludes the proof.
\end{proof}

We will write from now on
$$
\sigma_d\coloneq \sigma_{d}(y_1,\dots ,y_{2p-2})\in \F_p[y_1,\dots ,y_{2p-2}] .
$$

Motivated  by Lemma \ref{lem:symmetric-functions-2} and the proof of Theorem \ref{thm:K_2p+1}, we make the following ``Ansatz'':
\begin{align}\label{def:gL0,gR0}
g_{L_0}\coloneq 0\qquad \text{and}\qquad g_{R_0}\coloneq \sigma_p-l\cdot \sigma_{p-1}
\end{align}
for some linear polynomial 
$$
l=\sum_{i=1}^{2p-2} a_ix_i\in \F_p[x_1,\dots ,x_{2p-2}].
$$
By Lemma \ref{lem:symmetric-functions-2}, $\Delta_{e_{00}}(\sigma_{p-1})=0$.
By the Leibniz rule (see Lemma \ref{lem:Leibniz-rule}), we thus find
$$
\Delta_{e_{00}}(l\cdot \sigma_{p-1})=\Delta_{e_{00}}(l)\cdot \sigma_{p-1}.
$$
By Lemma \ref{lem:symmetric-functions-2}, $\Delta_{e_{00}}(\sigma_p)=1-\sigma_{p-1}$ and so we get
\begin{align} \label{eq:Delta-gR0:1}
\Delta_{e_{00}}(g_{R_0}-g_{L_0})=1-\sigma_{p-1}-\Delta_{e_{00}}(l)\sigma_{p-1} ,
\end{align}
where $\Delta_{e_{00}}(l)= \sum_{i=1}^{2p-2} a_i$.
The constants $a_i\in \F_p$ in the Ansatz \eqref{def:gL0,gR0} should therefore satisfy the following linear condition:
\begin{align} \label{eq:condition-ai}
 \sum_{i=1}^{2p-2} a_i=-1.
\end{align}
We thus arrive at the following

\begin{lemma} \label{lem:Delta-gR0}
Let $g_{L_0},g_{R_0}$ be as in \eqref{def:gL0,gR0} and $a_i$ as in \eqref{eq:condition-ai}. 
Then the following holds:
\begin{enumerate}
\item $\Delta_{e_{00}}(g_{R_0}-g_{L_0})= 1 $;\label{item:lem:Delta-gR0:1}
\item $\Delta_{e_{10}}(g_{R_0})=-\sigma_{p-1}(y_1,\dots ,y_{2p-2})$;\label{item:lem:Delta-gR0:2}
\item $\Delta_{e_{20}}(g_{R_0})=\sigma_{p-1}(y_1-1,\dots ,y_{2p-2}-1)-1$;\label{item:lem:Delta-gR0:3}
\end{enumerate} 
\end{lemma}
\begin{proof}
We use the description of the operators $\Delta_{e_{i0}}$ on the variables $x_i,y_i$ given in \eqref{eq:Delta-actions}.

The first item follows from \eqref{eq:Delta-gR0:1}.


We have
$
\Delta_{e_{10}}(\sigma_d)=0
$
for $d=p-1,p$, because $\sigma_d$ does not depend on $x_i$.
Hence,
$$
\Delta_{e_{10}}(\sigma_p-l\cdot \sigma_{p-1})=-(\sum_{i=1}^{2p-2}-a_i)\cdot \sigma_{p-1}=-\sigma_{p-1}.
$$
This proves item \eqref{item:lem:Delta-gR0:2}.

Finally, we have $\Delta_{e_{20}}l=0$ and
\begin{align*}
\Delta_{e_{20}} \sigma_p&=\sigma_p(y_1-1,\dots ,y_{2p-2}-1)-\sigma_p(y_1,\dots ,y_{2p-2}) \\
&= -1+\sigma_{p-1}(y_1-1,\dots ,y_{2p-2}-1),
\end{align*}
where the last equality follows from Lemma \ref{lem:symmetric-functions-2}.
This proves item \eqref{item:lem:Delta-gR0:3}.
\end{proof}

Lemma \ref{lem:Delta-gR0} yields the conditions
\begin{align} \label{eq:condition-Deta_ei0-gLi}
\Delta_{e_{10}}(g_{L_1})=-\sigma_{p-1}(y_1,\dots ,y_{2p-2})\qquad \text{and}\qquad \Delta_{e_{20}}(g_{L_2})=\sigma_{p-1}(y_1-1,\dots ,y_{2p-2}-1)-1.
\end{align}
As we have seen above, the second condition is satisfied by 
\begin{align} \label{def:gL2}
g_{L_2}\coloneq \sigma_{p}(y_1,\dots ,y_{2p-2}) .
\end{align}

\begin{lemma} \label{lem:gLi-gRj-conditions}
Let $g_{L_i},g_{R_j}\colon W\to \F_p$ with $0\leq i\leq 2$ and $0\leq j\leq 2p-2$ be polynomial test functions, satisfying \eqref{def:gL0,gR0} and \eqref{def:gL2}.
Then the conclusion of Theorem \ref{thm:existence-of-test-functions-K_3,2p-1} holds true if and only if the following hold:
\begin{enumerate}
\item $\Delta_{e_{0j}} (g_{R_j})=0 $ for all $1\leq j\leq 2p-2$;\label{item:lem:gLi-gRj-conditions:1}
\item $\Delta_{e_{1j}} (g_{R_j})=\Delta_{e_{1j}} (g_{L_1}) $ for all $1\leq j\leq 2p-2$;\label{item:lem:gLi-gRj-conditions:2}
\item $\Delta_{e_{2j}} (g_{R_j})=\sigma_{p-1}(y_1,\dots , \widehat y_j ,\dots ,y_{2p-2}) $ for all $1\leq j\leq 2p-2$;\label{item:lem:gLi-gRj-conditions:3}
\item $\Delta_{e_{10}}(g_{L_1})=-\sigma_{p-1}(y_1,\dots ,y_{2p-2})$. \label{item:lem:gLi-gRj-conditions:4}
\end{enumerate}
\end{lemma}
\begin{proof}
We first show that the conditions are necessary.
Item \eqref{item:lem:gLi-gRj-conditions:1} follows from $g_{L_0}=0$ and item \eqref{item:lem:gLi-gRj-conditions:2} is clear.
Item \eqref{item:lem:gLi-gRj-conditions:3} follows from the condition
$$
0=\Delta_{e_{2j}} (g_{R_j}-g_{L_2})=\Delta_{e_{2j}} (g_{R_j})-\Delta_{e_{2j}} (\sigma_p(y_1,\dots ,y_{2p-2}))
$$
together with (an obvious variant of) equation \eqref{eq:Delta-sigma_p=hat-sigma-p-1}, which proves
\begin{align*}
\Delta_{e_{2j}} (\sigma_p(y_1,\dots ,y_{2p-2})) 
&=\sigma_{p}(y_1,\dots , y_j+1,\dots ,y_{2p-2})-\sigma_{p}(y_1,\dots ,y_{2p-2}) \\
&=\sigma_{p-1}(y_1,\dots , \widehat y_j ,\dots ,y_{2p-2}) .
\end{align*}
Item \eqref{item:lem:gLi-gRj-conditions:4} follows from \eqref{eq:condition-Deta_ei0-gLi}.

For the converse, assume that item \eqref{item:lem:gLi-gRj-conditions:1}--\eqref{item:lem:gLi-gRj-conditions:4} hold true.
Then $\Delta_{e_{00}}(g_{R_0}-g_{L_0})= 1 $ by Lemma \ref{lem:Delta-gR0}.
Moreover, for $1\leq j\leq 2p-2$, we have $\Delta_{e_{0j}}(g_{R_j}-g_{L_0})=\Delta_{e_{0j}}(g_{R_j})=0 $ by item \eqref{item:lem:gLi-gRj-conditions:1}.
Next, for $1\leq j\leq 2p-2$ and $i=1,2$, we have $\Delta_{e_{ij}}(g_{R_j}-g_{L_i})=0 $ by item \eqref{item:lem:gLi-gRj-conditions:2} and \eqref{item:lem:gLi-gRj-conditions:3}, together with the above computation of $\Delta_{e_{2j}} (g_{L_2})$, where  $g_{L_2}=\sigma_p$ is given by \eqref{def:gL2}.
Finally, $\Delta_{e_{i0}}(g_{R_0}-g_{L_i)})=0$ for $i=1,2$ by \eqref{eq:condition-Deta_ei0-gLi}, \eqref{def:gL2} and item \eqref{item:lem:gLi-gRj-conditions:4}.
This concludes the proof of the lemma.
\end{proof}

For $1\leq j\leq 2p-2$, we define
$$
h_j\coloneq \sigma_{p-1}(y_1,\dots,\widehat y_j,\dots ,y_{2p-2})\in \F_p[y_1,\dots ,y_{2p-2}] .
$$
Item \eqref{item:lem:gLi-gRj-conditions:3} of Lemma \ref{lem:gLi-gRj-conditions} then reads as follows:
$$
\Delta_{e_{2j}} (g_{R_j})=h_j\qquad \text{for all $1\leq j\leq 2p-2$.}
$$
Moreover, $\Delta_{e_{2j}}(h_j)=0$ because $h_j$ does not depend on $y_j$.

Theorem \ref{thm:existence-of-test-functions-K_3,2p-1} is then a consequence of the following:

\begin{lemma}
Let $g_{L_i},g_{R_j}\colon W\to \F_p$ with $0\leq i\leq 2$ and $0\leq j\leq 2p-2$ be the polynomial test functions given by \eqref{def:gL0,gR0}, \eqref{def:gL2},
\begin{align} \label{def:gRj}
g_{R_j}\coloneq (y_j-x_j)h_j\qquad \text{for all $1\leq j\leq 2p-2$,}
\end{align}
and
\begin{align} \label{def:gL1}
g_{L_1}\coloneq -\sum_{j=1}^{2p-2}x_jh_j.
\end{align}
Then the conclusion of Theorem \ref{thm:existence-of-test-functions-K_3,2p-1} holds true.
\end{lemma}
\begin{proof}
It suffices to check items \eqref{item:lem:gLi-gRj-conditions:1}--\eqref{item:lem:gLi-gRj-conditions:4} in Lemma \ref{lem:gLi-gRj-conditions}.
By Lemma \ref{lem:e_ij-2}, we have $e_{0j}=-e_{1j}-e_{2j}$.
Using this, we see that $\Delta_{e_{0j}}(h_j)=0 $ and $\Delta_{e_{0j}}(y_j-x_j)=0 $.
Hence, the Leibniz rule (see Lemma \ref{lem:Leibniz-rule}) implies
$\Delta_{e_{0j}}(g_{R_j})=0$, which gives item \eqref{item:lem:gLi-gRj-conditions:1} in Lemma \ref{lem:gLi-gRj-conditions}. 

Item \eqref{item:lem:gLi-gRj-conditions:2} follows from the fact that the terms of $g_{R_j}$ and $g_{L_1}$ that involve $x_j$ agree.  

Using the Leibniz rule and the fact that $\Delta_{e_{2j}}h_j=0$, we find
$$
\Delta_{e_{2j}}g_{R_j}=\Delta_{e_{2j}}(y_j-x_j)\cdot h_j=h_j.
$$
This proves item \eqref{item:lem:gLi-gRj-conditions:3}.

Finally, to prove item \eqref{item:lem:gLi-gRj-conditions:4}, we use Lemma \ref{lem:e_ij-2}, which gives $e_{10}=-\sum_{j=1}^{2p-2}e_{1j}$.
Using this, we get
$$
\Delta_{e_{10}}(g_{L_1})=-\sum_{j=1}^{2p-2}\Delta_{e_{10}}( x_jh_j) = \sum_{j=1}^{2p-2} h_j  .
$$
The lemma follows now from the fact that 
$$
 \sum_{j=1}^{2p-2} h_j = \sum_{j=1}^{2p-2} \sigma_{p-1}(y_1,\dots,\widehat y_j,\dots ,y_{2p-2})=-\sigma_{p-1}(y_1, \dots ,y_{2p-2}).
$$
Indeed, for a subset $J\subset \{1,2,\dots ,2p-2\}$ of cardinality $|J|=p-1$, the monomial $\prod_{j\in J} y_j$ appears with the following coefficient in the above left hand side:
$$
\# (\{1,2,\dots ,2p-2\}\setminus J)=p-1\equiv -1\mod p.
$$
This concludes the proof of the lemma.
\end{proof}

\section*{Acknowledgments} 
Thanks for comments to O.~de Gaay Fortman.

ChatGPT was used for exploratory discussion, to aid in the construction of test functions, and for accelerating parts of the drafting process.  
No statement or proof was accepted on the basis of AI output; all mathematical content was independently verified by the authors, who take full responsibility for the manuscript.

This project has received funding from the European Research Council (ERC) under the European Union’s Horizon 2020 research and innovation programme under grant agreement No 948066 (ERC-StG RationAlgic).
The research was partly conducted in the framework of the DFG-funded research training group RTG 2965: From Geometry to Numbers, Project number 512730679.

\end{document}